\documentclass[10pt, reqno]{amsart}
\usepackage{amsmath, amsthm, amscd, amsfonts, amssymb, graphicx, color}
\usepackage[bookmarksnumbered, colorlinks, plainpages]{hyperref}
\hypersetup{colorlinks=true,linkcolor=red, anchorcolor=green, citecolor=cyan, urlcolor=red, filecolor=magenta, pdftoolbar=true}
\usepackage{mathrsfs}
\usepackage[OT2,OT1]{fontenc}
\newcommand\cyr
{
\renewcommand\rmdefault{wncyr}
\renewcommand\sfdefault{wncyss}
\renewcommand\encodingdefault{OT2}
\normalfont
\selectfont
}
\DeclareTextFontCommand{\textcyr}{\cyr}

\newtheorem{theorem}{Theorem}[section]
\newtheorem{lemma}[theorem]{Lemma}

\newtheorem{corollary}[theorem]{Corollary}
\theoremstyle{definition}
\newtheorem{definition}[theorem]{Definition}

\newtheorem{conjecture}[theorem]{Conjecture}

\theoremstyle{remark}
\newtheorem{remark}[theorem]{Remark}
\numberwithin{equation}{section}

\begin{document}
\setcounter{page}{1}

\title[%Galois extensions
non-abelian class field theory]{%On generators of the Galois extensions
Non-abelian class field theory and higher dimensional noncommutative tori}

\author[Nikolaev]
{Igor V. Nikolaev$^1$}

\address{$^{1}$ Department of Mathematics and Computer Science, St.~John's University, 8000 Utopia Parkway,  
New York,  NY 11439, United States.}
\email{\textcolor[rgb]{0.00,0.00,0.84}{igor.v.nikolaev@gmail.com}}

%\dedicatory{In memory of Ola Bratteli}



%\date{Received:  August 14, 2015; Revised: yyyyyy; Accepted: zzzzzz.}

\begin{abstract}
We study  a relation between the Drinfeld modules and the even dimensional 
noncommutative tori. A non-abelian class field theory is developed based on 
this relation. Explicit generators of the Galois extensions are constructed.   
\end{abstract}

\maketitle

%**************************************************************************
\section{Introduction}
%***************************************************************************
The class field theory studies abelian extensions of the number fields. 
Such a theory devoid of the $L$-functions is due to Claude Chevalley. 
The non-abelian class field theory is looking for a generalization 
to  the  extensions with non-commutative Galois groups $G$. The Artin $L$-function
associated to the group $G$ was a step forward to solve the problem  [Artin 1924] \cite{Art1}. 
The influential Langlands Program predicts  an analytic solution based on the $L$-functions
associated to the irreducible representations of the algebraic groups over adeles
[Langlands 1978] \cite{Lan1}.  Little is known about the non-abelian class field theory  
in the spirit of Chevalley.

The class field theory for function fields is due to  [Carlitz 1938] \cite{Car1} and  [Ore 1933] \cite{Ore1}.
The non-abelian counterpart was developed  by Drinfeld  in terms of  elliptic modules [Drinfeld 1974] \cite{Dri1}. 
The latter are now called the Drinfeld modules. 
It is remarkable that in both cases the generators of the Galois extensions are constructed 
explicitly, {\it ibid.} 

The aim of our note is a non-abelian class field theory for number fields 
based on  a representation of the Drinfeld modules   by the $C^*$-algebras of 
bounded linear operators on a Hilbert space. 
Such algebras,  known as noncommutative tori, classify the 
Drinfeld modules and therefore can be used in a non-abelian class field theory
for number fields (Theorem \ref{thm1.1}).  Moreover, the generators of such extensions
 are given explicitly (Corollary \ref{cor1.2}). 
This fact  explains a phenomenon  observed in \cite{Nik3}.

The paper is organized as follows.  A brief review of the preliminary facts is 
given in Section 2. Main results Theorem \ref{thm1.1} and Corollary \ref{cor1.2}  are 
formulated in Section 3.  They are proved in Section 4.
The Langlands reciprocity for the noncommutative tori  
is discussed in Section 5.

%**************************************************************************
\section{Preliminaries}
%***************************************************************************
We briefly review the noncommutative tori,  semigroup and cluster $C^*$-algebras, and Drinfeld modules.
We refer the reader to [Rieffel 1990] \cite{Rie1},  
[Li 2017] \cite{Li1},  \cite{Nik2},  \cite[Section 1.1]{N} and [Rosen 2002] \cite[Chapters 12 \& 13]{R} 
for a detailed exposition.  

%**************************************************************************
\subsection{Noncommutative tori}
%***************************************************************************
%**************************************************************************
\subsubsection{$C^*$-algebras}
%***************************************************************************
The $C^*$-algebra is an algebra  $\mathscr{A}$ over $\mathbf{C}$ with a norm 
$a\mapsto ||a||$ and an involution $\{a\mapsto a^* ~|~ a\in \mathscr{A}\}$  such that $\mathscr{A}$ is
complete with  respect to the norm, and such that $||ab||\le ||a||~||b||$ and $||a^*a||=||a||^2$ for every  $a,b\in \mathscr{A}$.  
Each commutative $C^*$-algebra is  isomorphic
to the algebra $C_0(X)$ of continuous complex-valued
functions on some locally compact Hausdorff space $X$. 
Any other  algebra $\mathscr{A}$ can be thought of as  a noncommutative  
topological space. 

%**************************************************************************
\subsubsection{K-theory of $C^*$-algebras}
%**********************************************************************
By $M_{\infty}(\mathscr{A})$ 
one understands the algebraic direct limit of the $C^*$-algebras 
$M_n(\mathscr{A})$ under the embeddings $a\mapsto ~\mathbf{diag} (a,0)$. 
The direct limit $M_{\infty}(\mathscr{A})$  can be thought of as the $C^*$-algebra 
of infinite-dimensional matrices whose entries are all zero except for a finite number of the
non-zero entries taken from the $C^*$-algebra $\mathscr{A}$.
Two projections $p,q\in M_{\infty}(\mathscr{A})$ are equivalent, if there exists 
an element $v\in M_{\infty}(\mathscr{A})$,  such that $p=v^*v$ and $q=vv^*$. 
The equivalence class of projection $p$ is denoted by $[p]$.   
We write $V(\mathscr{A})$ to denote all equivalence classes of 
projections in the $C^*$-algebra $M_{\infty}(\mathscr{A})$, i.e.
$V(\mathscr{A}):=\{[p] ~:~ p=p^*=p^2\in M_{\infty}(\mathscr{A})\}$. 
The set $V(\mathscr{A})$ has the natural structure of an abelian 
semi-group with the addition operation defined by the formula 
$[p]+[q]:=\mathbf{diag}(p,q)=[p'\oplus q']$, where $p'\sim p, ~q'\sim q$ 
and $p'\perp q'$.  The identity of the semi-group $V(\mathscr{A})$ 
is given by $[0]$, where $0$ is the zero projection. 
By the $K_0$-group $K_0(\mathscr{A})$ of the unital $C^*$-algebra $\mathscr{A}$
one understands the Grothendieck group of the abelian semi-group
$V(\mathscr{A})$, i.e. a completion of $V(\mathscr{A})$ by the formal elements
$[p]-[q]$.  The image of $V(\mathscr{A})$ in  $K_0(\mathscr{A})$ 
is a positive cone $K_0^+(\mathscr{A})$ defining  the order structure $\le$  on the  
abelian group  $K_0(\mathscr{A})$. The pair   $\left(K_0(\mathscr{A}),  K_0^+(\mathscr{A})\right)$
is known as a dimension group of the $C^*$-algebra $\mathscr{A}$.

%**************************************************************************
\subsubsection{Noncommutative tori}
%**********************************************************************
The $m$-dimensional noncommutative torus $\mathscr{A}_{\Theta}^m$ is the
universal $C^*$-algebra  generated by unitary operators $u_1,\dots, u_m$
satisfying the commutation relations 
%******************************************************************
\begin{equation}\label{eq1.1}
u_ju_i=e^{2\pi i \theta_{ij}} u_iu_j, \quad 1\le i,j\le m
\end{equation}
%*****************************************************************
for a skew-symmetric matrix  $\Theta=(\theta_{ij})\in M_m(\mathbf{R})$
[Rieffel 1990] \cite{Rie1}. 
 It is known that 
  $K_0(\mathscr{A}_{\Theta}^m)\cong K_1(\mathscr{A}_{\Theta}^m)\cong \mathbf{Z}^{2^{m-1}}$.
The canonical trace $\tau$ on the $C^*$-algebra
$\mathscr{A}_{\Theta}^m$ defines a homomorphism from 
$K_0(\mathscr{A}_{\Theta}^m)$ to the real line $\mathbf{R}$;
under the homomorphism, the image of $K_0(\mathscr{A}_{\Theta}^m)$
is a $\mathbf{Z}$-module, whose generators $\tau=(\tau_i)$ are polynomials 
in $\theta_{ij}$.  The noncommutative
tori  $\mathscr{A}_{\Theta}^m$ and $\mathscr{A}_{\Theta'}^m$ are Morita
equivalent,  if  the matrices $\Theta$ and $\Theta'$
belong to the same orbit of a subgroup $SO(m,m~|~\mathbf{Z})$ of the
group $GL_{2m}(\mathbf{Z})$, which acts on $\Theta$ by the formula
$\Theta'=(A\Theta+B)~/~(C\Theta+D)$, where $(A, B,  C,  D)\in GL_{2m}(\mathbf{Z})$
and  the matrices $A,B,C,D\in GL_m(\mathbf{Z})$ satisfy the conditions
%*******************************************************************
%\begin{equation}%\label{eq1}
$A^tD+C^tB=I,\quad A^tC+C^tA=0=B^tD+D^tB$,
%\end{equation}
%********************************************************************  
where $I$ is the unit matrix and $t$ at the upper right of a matrix 
means a transpose of the matrix.)  
The group $SO(m, m ~| ~\mathbf{Z})$ can be equivalently defined as a
subgroup of the group  $SO(m, m ~| ~\mathbf{R})$ consisting of linear transformations 
of the space $\mathbf{R}^{2m}$,  which 
preserve the quadratic form $x_1x_{m+1}+x_2x_{k+2}+\dots+x_kx_{2m}$.

%**************************************************************************
\subsubsection{AF-algebras}
%**********************************************************************
An {\it AF-algebra}  (Approximately Finite-dimensional $C^*$-algebra) is defined to
be the  norm closure of an ascending sequence of   finite dimensional
$C^*$-algebras $M_n$,  where  $M_n$ is the $C^*$-algebra of the $n\times n$ matrices
with entries in $\mathbf{C}$. Here the index $n=(n_1,\dots,n_k)$ represents
the  semi-simple matrix algebra  $M_i=M_{i_1}\oplus\dots\oplus M_{i_k}$.
The ascending sequence mentioned above  can be written as 
%***********************************************************
%\begin{equation}%\label{eq2.2}
\displaymath
M_{i_1}\buildrel\rm\varphi_1\over\longrightarrow M_{i_2}
   \buildrel\rm\varphi_2\over\longrightarrow\dots,
 \enddisplaymath  
%\end{equation}
%****************************************************
where $M_{i_k}$ are the finite dimensional $C^*$-algebras and
$\varphi_i$ the homomorphisms between such algebras.  
If $\varphi_i=Const$, then the AF-algebra $\mathscr{A}$ is called 
{\it stationary}. 
The homomorphisms $\varphi_i$ can be arranged into  a graph as follows. 
Let  $M_i=M_{i_1}\oplus\dots\oplus M_{i_k}$ and 
$M_{i'}=M_{i_1'}\oplus\dots\oplus M_{i_k'}$ be 
the semi-simple $C^*$-algebras and $\varphi_i: M_i\to M_{i'}$ the  homomorphism. 
One has  two sets of vertices $V_{i_1},\dots, V_{i_k}$ and $V_{i_1'},\dots, V_{i_k'}$
joined by  $a_{rs}$ edges  whenever the summand $M_{i_r}$ contains $a_{rs}$
copies of the summand $M_{i_s'}$ under the embedding $\varphi_i$. 
As $i$ varies, one obtains an infinite graph called the   Bratteli diagram of the
AF-algebra.  The matrix $A=(a_{rs})$ is known as  a  partial multiplicity matrix;
an infinite sequence of $A_i$ defines a unique AF-algebra.
If   $\mathbb{A}$ is a stationary AF-algebra, then   $A_i=Const$
for all $i\ge 1$.  
The  dimension group $\left(K_0(\mathbb{A}),  K_0^+(\mathbb{A})\right)$  is a complete invariant of the
Morita equivalence class of the AF-algebra $\mathbb{A}$, see e.g. \cite[Theorem 3.5.2]{N}.

%**************************************************************************
\subsubsection{Semigroup $C^*$-algebras}
%**********************************************************************
Let $G$ be a semigroup. We assume that $G$ is left cancellative, i.e. for all 
$g,x,y\in G$ the equality  $gx=gy$ implies $x=y$. In other words, the map $G\to G$ 
given by left multiplication is injective for all $g\in G$. 
Consider the left regular representation of $G$. Namely, let $\ell^2G$ be 
the Hilbert space with the canonical orthonormal basis $\{\delta_x : x\in G\}$,
where $\delta_x$ is the delta-function. For every $g\in G$ the map 
$G\to G$ given by the formula $x\mapsto gx$ is injective, so that the 
mapping $\delta_x\mapsto \delta_{gx}$ extends to an isometry 
$V_g:   \ell^2G\to \ell^2G$. In other words, the assignment $g\mapsto V_g$ 
represents $G$ as isometries of the space  $\ell^2G$. 
%******************************************************
\begin{definition}\label{dfn2.1}
The semigroup $C^*$-algebra $C^*(G)$ is the smallest subalgebra of bounded 
linear operators on the Hilbert space  $\ell^2G$ containing all $\{V_g : g\in G\}$.
In other words,
%**********************************************
%\begin{equation}
\displaymath
C^*(G):= C^*\left(\{V_g : g\in G\}\right). 
\enddisplaymath
%\end{equation}
%***********************************************
\end{definition}
%*******************************************************

%**************************************************************************
\subsection{Cluster $C^*$-algebras}
%***************************************************************************
The cluster algebra  of rank $n$ 
is a subring  $\mathcal{A}(\mathbf{x}, B)$  of the field  of  rational functions in $n$ variables
depending  on  variables  $\mathbf{x}=(x_1,\dots, x_n)$
and a skew-symmetric matrix  $B=(b_{ij})\in M_n(\mathbf{Z})$.
The pair  $(\mathbf{x}, B)$ is called a  seed.
A new cluster $\mathbf{x}'=(x_1,\dots,x_k',\dots,  x_n)$ and a new
skew-symmetric matrix $B'=(b_{ij}')$ is obtained from 
$(\mathbf{x}, B)$ by the   exchange relations [Williams 2014]  \cite[Definition 2.22]{Wil1}:
%*********************************************************************************************
\begin{eqnarray}%\label{eq2.3}
x_kx_k'  &=& \prod_{i=1}^n  x_i^{\max(b_{ik}, 0)} + \prod_{i=1}^n  x_i^{\max(-b_{ik}, 0)},\cr \nonumber
b_{ij}' &=& 
\begin{cases}
-b_{ij}  & \mbox{if}   ~i=k  ~\mbox{or}  ~j=k\cr
b_{ij}+{|b_{ik}|b_{kj}+b_{ik}|b_{kj}|\over 2}  & \mbox{otherwise.}
\end{cases}
\end{eqnarray}
%******************************************************************************************* 
The seed $(\mathbf{x}', B')$ is said to be a  mutation of $(\mathbf{x}, B)$ in direction $k$.
where $1\le k\le n$.  The  algebra  $\mathcal{A}(\mathbf{x}, B)$ is  generated by the 
cluster  variables $\{x_i\}_{i=1}^{\infty}$
obtained from the initial seed $(\mathbf{x}, B)$ by the iteration of mutations  in all possible
directions $k$.   The  Laurent phenomenon
 says  that  $\mathcal{A}(\mathbf{x}, B)\subset \mathbf{Z}[\mathbf{x}^{\pm 1}]$,
where  $\mathbf{Z}[\mathbf{x}^{\pm 1}]$ is the ring of  the Laurent polynomials in  variables $\mathbf{x}=(x_1,\dots,x_n)$
 [Williams 2014]  \cite[Theorem 2.27]{Wil1}.
In particular, each  generator $x_i$  of  the algebra $\mathcal{A}(\mathbf{x}, B)$  can be 
written as a  Laurent polynomial in $n$ variables with the   integer coefficients.

 The cluster algebra  $\mathcal{A}(\mathbf{x}, B)$  has the structure of an additive abelian
semigroup consisting of the Laurent polynomials with positive coefficients. 
In other words,  the $\mathcal{A}(\mathbf{x}, B)$ is a dimension group, see Section 2.1.6 or  
\cite[Definition 3.5.2]{N}.
The cluster $C^*$-algebra  $\mathbb{A}(\mathbf{x}, B)$  is   an  AF-algebra,  
such that $K_0(\mathbb{A}(\mathbf{x}, B))\cong  \mathcal{A}(\mathbf{x}, B)$.

%**************************************************************************
\subsubsection{Cluster $C^*$-algebra $\mathbb{A}(S_{g,n})$}
%***************************************************************************
Denote by $S_{g,n}$  the Riemann surface   of genus $g\ge 0$  with  $n\ge 0$ cusps.
 Let   $\mathcal{A}(\mathbf{x},  S_{g,n})$ be the cluster algebra 
 coming from  a triangulation of the surface $S_{g,n}$   [Williams 2014]  \cite[Section 3.3]{Wil1}. 
 We shall denote by  $\mathbb{A}(S_{g,n})$  the corresponding cluster $C^*$-algebra. 
 Let $p$ be a prime number, and denote by  $\mathcal{A}_p(S_{g,n})$ a sub-algebra of  $\mathcal{A}(S_{g,n})$
consisting of the Laurent polynomials whose coefficients are divisible by $p$. It is easy to verify that   $\mathcal{A}_p(S_{g,n})$
is again a dimension group under the addition of the Laurent polynomials. We say that  $\mathbb{A}_p(S_{g,n})$ is a
congruence sub-algebra of level $p$ of the cluster $C^*$-algebra  $\mathbb{A}(S_{g,n})$, i.e. 
$K_0(\mathbb{A}_p(S_{g,n}))\cong \mathcal{A}_p(S_{g,n})$.

Let $T_{g,n}$ be the Teichm\"uller space of the surface $S_{g,n}$,
i.e. the set of all complex structures on $S_{g,n}$ endowed with the 
natural topology. The geodesic flow $T^t: T_{g,n}\to T_{g,n}$
is a one-parameter  group of matrices $\left(\small\begin{matrix} e^t &0\cr 0 &e^{-t}\end{matrix}\right)$
acting on the holomorphic quadratic differentials on the Riemann surface $S_{g,n}$. 
Such a flow gives rise to a one parameter group of automorphisms 
%****************************************************************
\begin{equation}\label{eq2.4}
\sigma_t: \mathbb{A}(S_{g,n})\to \mathbb{A}(S_{g,n})
\end{equation}
%*************************************************************
called the Tomita-Takesaki flow on the AF-algebra $\mathbb{A}(S_{g,n})$. 
Denote by $Prim~\mathbb{A}(S_{g,n})$ the space of all primitive ideals 
of $\mathbb{A}(S_{g,n})$ endowed with the Jacobson topology. 
Recall (\cite{Nik2}) that each primitive ideal has a parametrization by a vector 
$\Theta\in \mathbf{R}^{6g-7+2n}$ and we write it 
$I_{\Theta}\in Prim~\mathbb{A}(S_{g,n})$
%****************************************************************************
\begin{theorem}\label{thm2.2}
{\bf (\cite{Nik2})}
There exists a homeomorphism
%**************************************************************************
%\begin{equation}
$h:  Prim~\mathbb{A}(S_{g,n})\times \mathbf{R}\to \{U\subseteq  T_{g,n} ~|~U~\hbox{{\sf is generic}}\},$
%\end{equation}
%*****************************************************************************
where $h(I_{\Theta},t)=S_{g,n}$ is 
given by the formula $\sigma_t(I_{\Theta})\mapsto S_{g,n}$;  the set $U=T_{g,n}$ if and only if
$g=n=1$.   The $\sigma_t(I_{\Theta})$
is an ideal of  $\mathbb{A}(S_{g,n})$ for all $t\in \mathbf{R}$ and 
 the quotient  algebra $AF$-algebra  $\mathbb{A}(S_{g,n})/\sigma_t(I_{\Theta}):=\mathbb{A}_{\Theta}^{6g-6+2n}$
is  a non-commutative coordinate ring  of  the Riemann surface  $S_{g,n}$.  
\end{theorem}
%****************************************************************************
Let $\phi\in Mod ~(S_{g,n})$ be a pseudo-Anosov automorphism of $S_{g,n}$ 
with the dilatation $\lambda_{\phi}>1$. If the Riemann surfaces $S_{g,n}$ and $\phi (S_{g,n})$
lie on the axis of $\phi$, then Theorem \ref{thm2.2} gives rise  the Connes invariant \cite[Section 4.2]{Nik2}:
%**************************************************************************
\begin{equation}\label{eq2.6}
T(\mathbb{A}(S_{g,n}))=\{\log\lambda_{\phi} ~|~ \phi\in Mod ~(S_{g,n})\}. 
\end{equation}
%*****************************************************************************

%**************************************************************************
\subsubsection{Relation to noncommutative tori   $\mathscr{A}_{\Theta}^{6g-6+2n}$}
%***************************************************************************
It is known that each  noncommutative torus   $\mathscr{A}_{\Theta}^{6g-6+2n}$
is a crossed product $C^*$-algebra embedded into an $AF$-algebra 
$\mathbb{A}_{\Theta}^{6g-6+2n}$,  such that 
$K_0(\mathscr{A}_{\Theta}^{6g-6+2n})\cong K_0( \mathbb{A}_{\Theta}^{6g-6+2n})$
 [Putnam 1989] \cite[Theorem 6.7]{Put1}.
Therefore  matrix $\Theta$ takes  the form:
%******************************************************************
\begin{equation}\label{eq1.2}
\Theta=
\left(
\begin{matrix}
0 & \alpha_1 &  &\cr
-\alpha_1 & 0 & \alpha_2 & \cr
  & -\alpha_2 & 0  &  &\cr
  \vdots & & & \vdots\cr
             & & & 0 & \alpha_{6g-7+2n} \cr
             & &  & - \alpha_{6g-7+2n} &0
\end{matrix}
\right). 
\end{equation}
%*****************************************************************
The noncommutative torus $\mathscr{A}_{\Theta}^{6g-6+2n}$  is said to have real multiplication,   
if all $\alpha_k$ in  (\ref{eq1.2})  are algebraic numbers
 \cite[Section 1.1.5.2]{N}; we write $\mathscr{A}_{RM}^{6g-6+2n}$ in this case. 
Likewise,  one can think of $\alpha_k$ as components of 
 a normalized eigenvector $(1,\alpha_1,\dots,\alpha_{6g-7+2n})$ corresponding to the 
 Perron-Frobenius eigenvalue $\varepsilon>1$ of a positive matrix 
 $B\in GL_{6g-6+2n}(\mathbf{Z})$.

%**************************************************************************
\subsection{Drinfeld modules ([Rosen 2002] \cite[Section 12]{R})}
%***************************************************************************
The explicit class field theory for the function fields is strikingly simpler
than for the number fields. The generators of the maximal abelian unramified
extensions (i.e. the Hilbert class fields) are constructed using 
the concept of the Drinfeld module. Roughly speaking, such a module 
is an analog of the exponential function and a generalization of the Carlitz module. 
Nothing similar  exists at the number fields side, where  the explicit 
generators of abelian extensions are known only for the field of rationals
(Kronecker-Weber theorem)
and imaginary quadratic number fields (complex multiplication). 
Below we give some details
on the Drinfeld modules.

Let $k$ be a field.  A polynomial $f\in k[x]$ is said to be additive
in the ring $k[x,y]$ if $f(x+y)=f(x)+f(y)$. If $char ~k=p$, then it is verified 
directly that the polynomial $\tau(x)=x^p$ is additive. Moreover, each 
additive polynomial has the form $a_0x+a_1x^p+\dots+a_rx^{p^r}$. 
The set of all additive polynomials is closed under addition and composition 
operations. The corresponding ring is isomorphic to a ring $k\langle\tau_p\rangle$
of the non-commutative polynomials given by  the commutation relation:
%********************************************************************
\begin{equation}\label{eq2.7}
\tau_p a=a^p\tau_p, \qquad \forall a\in k. 
\end{equation}
%****************************************************************

Let $A=\mathbf{F}_q[T]$ and $k=\mathbf{F}_q(T)$. 
The  Drinfeld module $Drin_A^r(k)$  of rank $r\ge 1$ is a homomorphism
%********************************************************************
%\begin{equation}\label{eq1.3}
$\rho: A\buildrel r\over\longrightarrow k\langle\tau_p\rangle$
%\end{equation}
%****************************************************************
given by a polynomial $\rho_a=a+c_1\tau_p+c_2\tau_p^2+\dots+c_r\tau_p^r$ with $c_i\in k$ and $c_r\ne 0$, 
such that for all $a\in A$ the constant term of $\rho_a$ is $a$ and 
$\rho_a\not\in k$ for at least one $a\in A$.  
Consider the torsion submodule $\Lambda_{\rho}[a]:=\{\lambda\in\bar k ~|~\rho_a(\lambda)=0\}$ of
%a non-trivial  Drinfeld module  $Drin_A^r(k)$. 
the $A$-module $\bar k_{\rho}$. 
The following result implies  a non-abelian class field theory 
for the function fields. 
%*********************************************************************
\begin{theorem}\label{thm2.3}
{\bf ([Rosen 2002] \cite[Proposition 12.5]{R})}
For each non-zero $a\in A$ the function 
field $k\left(\Lambda_{\rho}[a]\right)$  is a Galois extension of $k$,
such that its  Galois group is isomorphic to a subgroup of the matrix group $GL_r\left(A/aA\right)$.
%,where   $\Lambda_{\rho}[a]$ is a torsion submodule of the non-trivial  Drinfeld module  $Drin_A^r(k)$. 
\end{theorem}
%*********************************************************************

%**************************************************************************
\section{Main results}
%***************************************************************************
Let  $Drin_A^r(k)$  be the  Drinfeld module of rank $r\ge 1$, 
i.e. a homomorphism
%********************************************************************
\begin{equation}\label{eq1.3}
\rho:  ~A\buildrel r\over\longrightarrow k\langle\tau_p\rangle
\end{equation}
%****************************************************************
given by a polynomial $\rho_a=a+c_1\tau_p+c_2\tau_p^2+\dots+c_r\tau_p^r$ with $c_i\in k$ and $c_r\ne 0$, 
such that for all $a\in A$ the constant term of $\rho_a$ is $a$ and 
$\rho_a\not\in k$ for at least one $a\in A$ [Rosen 2002] \cite[p. 200]{R}.
For each non-zero $a\in A$ the function 
field $k\left(\Lambda_{\rho}[a]\right)$  is a Galois extension of $k$,
such that its  Galois group is isomorphic to a subgroup $G$ of the matrix group $GL_r\left(A/aA\right)$,
where   $\Lambda_{\rho}[a]=\{\lambda\in\bar k ~|~\rho_a(\lambda)=0\}$
is a torsion submodule of the non-trivial  Drinfeld module  $Drin_A^r(k)$  [Rosen 2002] \cite[Proposition 12.5]{R}.
Clearly, the abelian extensions correspond to the case $r=1$.

Let $G$ be a  left cancellative  semigroup generated by $\tau_p$ and all  $a_i\in k$ subject to the commutation relations 
$\tau_p a_i=a_i^p\tau_p$.  Let $C^*(G)$ be the semigroup $C^*$-algebra [Li 2017] \cite{Li1}.  
For a Drinfeld module  $Drin_A^r(k)$  defined  by  (\ref{eq1.3}) we consider a homomorphism of the semigroup $C^*$-algebras:  
 %********************************************************************
\begin{equation}\label{eq1.4}
C^*(A)\buildrel r\over\longrightarrow C^*(k\langle\tau_p\rangle). 
\end{equation}
%****************************************************************
It is proved (Lemma \ref{lm3.1}) that there exists an $\alpha$ such that
$C^*(A)\subset I_{\alpha}$ and  $C^*(k\langle\tau_p\rangle)\subset\mathbb{A}_p(S_{g,n})$,
where  $\mathbb{A}_p(S_{g,n})$ is a congruence subalgebra 
of level $p$ of the cluster $C^*$-algebra $\mathbb{A}(S_{g,n}), ~I_{\alpha}$ is a primitive ideal of $\mathbb{A}_p(S_{g,n})$
and  $r=3g-3+n$, see Section 2.2.1. 
 An exact sequence of the ring homomorphisms $1\to A\buildrel \rho\over\to k\langle\tau_p\rangle\to k\langle\tau_p\rangle/\rho(A)\to 1$
 gives rise to such of the $C^*$-algebras  $1\to C^*(A)\buildrel \rho\over\to C^*(k\langle\tau_p\rangle)\to C^*(k\langle\tau_p\rangle)/C^*(\rho(A))\to 1.$
In view of the inclusion $C^*(\rho(A))\subset I_{\alpha}$ and the isomorphism    $\mathbb{A}_p(S_{g,n})/I_{\alpha}\cong \mathbb{A}_{RM}^{6g-6+2n}$,
the following map is well defined.
%****************************************************************
\begin{definition}
By $F: Drin_A^{3g-3+n}(k)\mapsto \mathbb{A}_{RM}^{6g-6+2n}$ one understands a map acting by the formula:
 %********************************************************************
\begin{equation}\label{eq3.3*}
\rho=(c_1,\dots,c_{3g-3+n})\mapsto (1,\alpha_1,\dots,\alpha_{6g-7+2n})=\alpha. 
\end{equation}
%****************************************************************
The map $F$ extends to the noncommutative tori  $\mathscr{A}_{RM}^{6g-6+2n}$ via the group 
isomorphism  $K_0( \mathbb{A}_{\Theta}^{6g-6+2n})\cong K_0(\mathscr{A}_{\Theta}^{6g-6+2n})$,
see Section 2.2.2.  Moreover, the map $F$ on the torsion submodule $\Lambda_{\rho}[a]$ 
is defined by the formula: 
 %********************************************************************
\begin{equation}\label{}
F(\Lambda_{\rho}[a])=\{\log ~(\varepsilon) ~e^{2\pi i\alpha_k} ~|~1\le k\le 2r-1\},
\end{equation}
%****************************************************************
where $\log ~(\varepsilon)$ is a scaling factor (\ref{eq2.6}) 
  coming from the geodesic flow on the Teichm\"uller space of surface $S_{g,n}$,
  see below.
 \end{definition}
%****************************************************************
%******************************************************************
\begin{remark}
Explicit formulae (\ref{eq3.3*}) for $\alpha_i$ as function of $c_i$ 
are known in a few cases, e.g. for the Carlitz modules  \cite{Nik3}.  
A general case is unknown to the author. 
\end{remark}
%*******************************************************************
 To recast the torsion submodule $\Lambda_{\rho}[a]$ in terms of $\mathscr{A}_{RM}^{2r}$, 
recall that if $r=1$ then the Drinfeld module (\ref{eq1.3}) plays the r\^ole of an elliptic curve  with complex multiplication  $\mathcal{E}_{CM}$
[Drinfeld 1974] \cite[p. 594]{Dri1}. The set $\Lambda_{\rho}[a]$ consists
of coefficients of the curve  $\mathcal{E}_{CM}$ which are equal to the value  of the Weierstrass $\wp$-function at the torsion points
of lattice $\Lambda_{CM}\subset\mathbf{C}$, where $\mathcal{E}_{CM}\cong \mathbf{C}/\Lambda_{CM}$ {\it ibid.}
Likewise, if $r\ge 1$ then the algebra $\mathscr{A}_{RM}^{2r}$ plays the r\^ole of   the curve  $\mathcal{E}_{CM}$
given by  coefficients  $e^{2\pi i \alpha_k}$ in the commutation equations (\ref{eq1.1}). 
 We  therefore get
 $F(\Lambda_{\rho}[a])=\{\log ~(\varepsilon) ~e^{2\pi i\alpha_k} ~|~1\le k\le 2r-1\}$, where $\log ~(\varepsilon)$ is a scaling factor (\ref{eq2.6}) 
  coming from the geodesic flow on the Teichm\"uller space of surface $S_{g,n}$ (Section 2.2.1). 
Our main results can be formulated as follows. 
%***************************************************************
\begin{theorem}\label{thm1.1}
Let $r=3g-3+n$. The following is true:

\medskip
(i) the map $F: Drin_A^{r}(k)\mapsto \mathscr{A}_{RM}^{2r}$ is a functor 
from the category of Drinfeld  modules $\mathfrak{D}$ to a category 
of the noncommutative tori $\mathfrak{A}$,   which maps any pair of isogenous  (isomorphic, resp.) 
modules  $Drin_A^{r}(k), ~\widetilde{Drin}_A^{r}(k)\in \mathfrak{D}$
to a pair of the homomorphic (isomorphic, resp.)  tori  $\mathscr{A}_{RM}^{2r}, \widetilde{\mathscr{A}}_{RM}^{2r}
\in \mathfrak{A}$;  

\smallskip
(ii) $F(\Lambda_{\rho}[a])=\{e^{2\pi i\alpha_k+\log\log\varepsilon} ~|~1\le k\le 2r-1\}$,
where $\mathscr{A}_{RM}^{2r}=F(Drin_A^r(k))$, 
$\alpha_i$ are given by (\ref{eq3.3*}),  $\log\varepsilon$ is the scaling factor (\ref{eq2.6})
 and $\Lambda_{\rho}(a)$ is the  torsion submodule of the $A$-module $\bar k_{\rho}$; 
%Drinfeld module  $Drin_A^{r}(k)$; 

\smallskip
(iii) the Galois group $Gal \left(\mathbf{k}(e^{2\pi i\alpha_k+\log\log\varepsilon})  ~| ~\mathbf{k}\right)\subseteq GL_{r}\left(A/aA\right)$,
where $\mathbf{k}$ is a subfield of the number field $\mathbf{Q}(e^{2\pi i\alpha_k+\log\log\varepsilon})$. 
 \end{theorem}
%***************************************************************
Theorem \ref{thm1.1} implies a non-abelian class field theory as follows.
Fix a non-zero $a\in A$ and let $G:=Gal~(k(\Lambda_{\rho}[a]) ~|~ k)\subseteq GL_r(A/aA)$,
where  $\Lambda_{\rho}[a]$ is the torsion submodule of the $A$-module $\bar k_{\rho}$.
%the Drinfeld module  $Drin_A^{r}(k)$. 
Consider the number field $\mathbf{K}=\mathbf{Q}(F(\Lambda_{\rho}[a]))$. 
Denote by $\mathbf{k}$ the maximal subfield of $\mathbf{K}$ which is fixed by the action of 
all elements of the group $G$. (The action of $G$ on $\mathbf{K}$ is well defined in view of item (ii) 
of Theorem \ref{thm1.1}.) 
%***************************************************************
\begin{corollary}\label{cor1.2} 
{\bf (Non-abelian class field theory)} 
The number field
%***********************************************
\begin{equation}\label{eq1.5}
\mathbf{K}\cong
\begin{cases} \mathbf{k}\left(e^{2\pi i\alpha_k +\log\log\varepsilon}\right), & if ~\mathbf{k}\subset\mathbf{C} - \mathbf{R},\cr
               \mathbf{k}\left(\cos 2\pi\alpha_k \times\log\varepsilon\right), & if ~\mathbf{k}\subset\mathbf{R},
\end{cases}               
\end{equation}
%***************************************************
is a Galois extension of  $\mathbf{k}$,
 such that  $Gal~(\mathbf{K} | \mathbf{k})\cong G$.
\end{corollary}
%***************************************************************
%*************************************************************
\begin{remark}
Formulas (\ref{eq1.5})  for  the abelian extensions of the
 imaginary quadratic fields $\mathbf{k}$ were  proved in \cite{Nik3}. 
An abelian class field theory  based on the Bost-Connes crossed product 
$C^*$-algebras was studied  in [Yalkinoglu  2013] \cite{Yal1}. 
\end{remark}
%**************************************************************

%**************************************************************************
\section{Proofs}
%***************************************************************************

%**************************************************************************
\subsection{Proof of Theorem \ref{thm1.1}}
%***************************************************************************
For the sake of clarity, let us outline the main ideas. 
For a Drinfeld module  $Drin_A^r(k)$ given by the formula (\ref{eq1.3}) 
we consider a homomorphism
$C^*(A)\buildrel r\over\longrightarrow C^*(k\langle\tau_p\rangle)$
of the $C^*$-algebras
defined by the multiplicative semigroups of the rings $A$ and  $k\langle\tau_p\rangle$,
respectively. 
 We prove that if $r=3g-3+n$ then $C^*(k_{\rho_a}\langle\tau_p\rangle)\subset \mathbb{A}_p(S_{g,n})$, where $k_{\rho_a}\langle\tau_p\rangle$ 
is the ring $k\langle\tau_p\rangle$ modulo the relation $T=\rho_a(T)$ and  $\mathbb{A}_p(S_{g,n})\subset \mathbb{A}(S_{g,n})$
is a congruence subalgebra of level $p$, see  Section 2.2.1.  
  Moreover,   $C^*(A)\subset I_{\alpha}$,  where 
$I_{\alpha}$ is a primitive ideal of the cluster $C^*$-algebra  $\mathbb{A}_p(S_{g,n})$
(Theorem \ref{thm2.2}). 
   Thus one gets a classification of the Drinfeld modules  $Drin_A^r(k)$
by the noncommutative tori $\mathscr{A}_{RM}^{2r}$ given by  (\ref{eq1.2});  this fact follows from Theorem \ref{thm2.2}
and an isomorphism  $\mathbb{A}_p(S_{g,n})/I_{\alpha}\cong \mathbb{A}_{RM}^{2r}$,  where $K_0( \mathbb{A}_{RM}^{2r})\cong K_0(\mathscr{A}_{RM}^{2r} )$
(Section 2.2.2). 
Let us pass to a detailed argument. 

\bigskip
(A)  Let us prove item (i) of Theorem \ref{thm1.1}.  We split the proof in a series of lemmas. 

%************************************************************
\begin{lemma}\label{lm3.1}
{\bf (Main lemma)}
 For any Drinfeld module  $Drin_A^{3g-3+n}(k)$ with $g\ge 1$ and $n\in\{0,1,2\}$ there exist 
 the following inclusions of the $C^*$-algebras:
 %********************************************
\begin{equation}
\begin{cases} 
C^*(k_{\rho_a}\langle\tau_p\rangle)\subset \mathbb{A}_p(S_{g,n})&\cr
C^*(\rho(A)) \subset I_{\alpha},&
\end{cases}
\end{equation}
%******************************************
 where  $I_{\alpha}$ is a primitive ideal of the cluster $C^*$-algebra  $\mathbb{A}_p(S_{g,n})$. 
\end{lemma}
%********************************************************
\begin{proof} 
Roughly speaking, the $C^*$-algebra  $C^*(k_{\rho_a}\langle\tau_p\rangle)$
has the structure of crossed product $C^*$-algebra by a minimal homeomorphism of the Cantor set
[Putnam 1989] \cite{Put1}.  Namely,  $C^*(k_{\rho_a}\langle\tau_p\rangle)\cong C^*(A)\rtimes_{\rho}\mathbf{Z}$, 
where  $C^*(A)$ is the semigroup $C^*$-algebra of $A$.
We apply Putnam-Vershik's Theorem on embedding of the crossed products into the $AF$-algebras
 [Putnam 1989] \cite[Theorem 6.7]{Put1}.
Let us pass  to a step-by-step construction.

\medskip
(i)  To define a crossed product $C^*$-algebra $C^*(A)\rtimes_{\rho}\mathbf{Z}$, 
consider a (semi-)dynamical system on $A$ given by the powers of $T$.
Namely,  the orbit of generator $T$ of the ring $A$ is given by  the set
$\{\rho_a(T), \rho_a(T^2),\dots, \rho_a(T^i),\dots\}$. One gets an orbit 
of any element of $A$ by the formula $\{F_q[\rho_a(T^i)]\}_{i=1}^{\infty}$. 
The  dynamical system on the $C^*$-algebra $C^*(A)$
gives rise to a crossed product $C^*$-algebra which we shall 
denote by  $C^*(A)\rtimes_{\rho}\mathbf{Z}$. It follows from the construction 
that $C^*(\rho(A))$ is the maximal abelian subalgebra of $C^*(A)\rtimes_{\rho}\mathbf{Z}$.

\smallskip
(ii) Denote by 
$(V,E, \le)$ the Bratteli diagram endowed with partial order $\le$ defined for edges adjacent 
to the same vertex. Recall that a Bratteli compactum $\mathscr{B}$ is the set of infinite paths $\{e_n\}$
of   $(V,E,\le)$ endowed with the distance $d(\{e_n\}, \{e_n'\})=2^{-k}$, where $k=\inf \{i ~|~e_i\ne e_i'\}$
[Herman, Putnam \& Skau 1992] \cite{HerPutSka1}.  The space $\mathscr{B}$ is a Cantor set. The Vershik map $h: \mathscr{B}\to\mathscr{B}$ 
is defined by the order $\le$ as follows. Let $e_n'$ be successor of $e_n$ relative the order $\le$.
The map  $h: \mathscr{B}\to\mathscr{B}$ is defined by the formula $\{e_n\}\mapsto \{e_1,\dots,e_{n}, e_{n+1}',
e_{n+2}',\dots\}$.  The map $h$ is a homeomorphism which generates a minimal dynamical system 
on the Cantor set $\mathscr{B}$.  Consider now an inclusion of the Bratteli compacta  $\mathscr{B}_{\alpha}\subset\mathscr{B}$
 corresponding to the inclusion of the AF-algebras $I_{\alpha}\subset\mathbb{A}_p(S_{g,n})$. 
The Vershik map $h_{\alpha}: \mathscr{B}_{\alpha}\to \mathscr{B}_{\alpha}$ extends  uniquely 
to a homeomorphism $h_{\alpha}: \mathscr{B}\to \mathscr{B}$. (Notice that $h_{\alpha}$ is no longer 
the Vershik map on $\mathscr{B}$.)  The minimal dynamical system on the Cantor set $\mathscr{B}$ defined 
by $h_{\alpha}$ gives rise to a crossed product $C^*$-algebra $C(\mathscr{B})\rtimes_{\alpha}\mathbf{Z}$,
where  $C(\mathscr{B})$ is the $C^*$-algebra of continuous complex-valued functions on $\mathscr{B}$. 
It follows from the construction that $I_{\alpha}$ is the maximal abelian subalgebra of 
 $C(\mathscr{B})\rtimes_{\alpha}\mathbf{Z}$. 
 
 \smallskip
 (iii) We proceed by establishing an isomorphism $C^*(A)\cong C(\mathscr{B})$. 
 Denote by $\hat A$ the Pontryagin dual of the multiplicative group of $A$. 
 Since $A$ is commutative,  the Gelfand theorem implies an isomorphism   $C^*(A)\cong C(\hat A)$,
where  $C(\hat A)$ is the $C^*$-algebra of continuous complex-valued functions on $\hat A$. 
Therefore it is sufficient to construct a homeomorphism between the topological spaces  $\hat A$ and $\mathscr{B}$. 

\smallskip
(iv) Let us calculate $\hat A$. Recall that $A$ is an analog of the ring $\mathbf{Z}$ and $\hat{\mathbf{Z}}\cong\mathbf{R}/\mathbf{Z}$.  
Likewise $\hat A\cong \mathbf{Q}_p/\mathbf{Z}_p$,  where $\mathbf{Q}_p$ and $\mathbf{Z}_p$ are the $p$-adic numbers and the $p$-adic
integers, respectively. 

\smallskip
(v) Let $\mathscr{B}$ be the Bratteli compactum of the AF-algebra $\mathbb{A}_p(S_{g,n})$. 
To describe the Bratteli diagram of  $\mathbb{A}_p(S_{g,n})$, let us recall the construction of
$\mathbb{A}(S_{1,1})$ using  the Farey tessellation of the hyperbolic plane [Boca 2008] \cite[Figure 2]{Boc1}. 
Consider the set of rational numbers  on the unit interval of the form $\mathbf{Z}[\frac{1}{p}]/\mathbf{Z}$. 
(Notice that such a set has structure of the Pr\"ufer group $\mathbf{Z}(p^{\infty})$ which we  will not use.)   
Consider a Bratteli diagram $(V_p,E_p)$ obtained by ``suspension of cusps'' of the Farey tessellation at the 
rational points $\mathbf{Z}(p^{\infty})$ of the unit interval $[0,1]$, see  [Boca 2008] \cite[Figures 1 \& 2]{Boc1}.
The  $(V_p,E_p)$ consists of all vertices at the vertical lines of the suspension and all subgraphs rooted in the 
above vertices. (Recall that all  graphs are  oriented by the level structure of the diagram.) It is verified directly that 
$(V_p,E_p)$ is the Bratteli diagram of the AF-algebra $\mathbb{A}_p(S_{1,1})$. The general case  $\mathbb{A}_p(S_{g,n})$
is treated similarly replacing the triangles in the Farey tessellation by the fundamental polygons of the surface $S_{g.n}$.  
The reader can see  that  the Bratteli compactum of $(V_p,E_p)$ is $\mathscr{B}=\mathbf{Z}[\frac{1}{p}]/\mathbf{Z}$. 

\smallskip
(vi)  It follows from items (iv) and (v) that $\hat A\cong \mathbf{Q}_p/\mathbf{Z}_p\cong \mathbf{Z}[\frac{1}{p}]/\mathbf{Z}=\mathscr{B}$.
(The middle isomorphism in the last formula is well known.)  Thus we conclude that
%********************************************************************
\begin{equation}\label{eq3.1}
C^*(A)\cong C(\hat A)\cong C(\mathscr{B}). 
\end{equation}
%*******************************************************************

\smallskip
(vii) To finish the proof,  we apply 
 Putnam-Vershik's Theorem on embedding of the crossed products into the $AF$-algebras
 [Putnam 1989] \cite[Section 6]{Put1}.
Namely, let $y\in \hat A$ and let  $\mathbb{A}_y$ be the $AF$-algebra, such that there 
exists a unital embedding  $C^*(A)\rtimes_{\rho}\mathbf{Z}\hookrightarrow  \mathbb{A}_y$
 [Putnam 1989] \cite[Theorem 6.7]{Put1}. It is immediate from item (v) that  $\mathbb{A}_y\subset \mathbb{A}_p(S_{g,n})$,
 since the Bratteli diagram of  $\mathbb{A}_y$ is a subgraph of  $(V_p,E_p)$ rooted in $y$. 
 In particular,  $C^*(k_{\rho_a}\langle\tau_p\rangle)\cong C^*(A)\rtimes_{\rho}\mathbf{Z}\subset  \mathbb{A}_p(S_{g,n})$
  and  $C^*(\rho(A)) \subset I_{\alpha}$.

\smallskip
(viii) Finally, let us show that the rank of the Drinfeld module (\ref{eq1.3}) is equal to $r=3g-3+n$. 
Indeed, recall that the Drinfeld module  (\ref{eq1.3})  is defined by the equation $T=\rho_a(T)$.
Roughly speaking, in item (vii) we identified $\rho$ with the (real) algebraic numbers $(1,\alpha_1,\dots,\alpha_{6g-7+2n})$
given by the minimal polynomial $z=\rho_a(z)$ having the complex roots $z_k=\alpha_k+i\alpha_{k+1}$. 
Since  $\rho_a$ has $r$ complex roots, one concludes that $2r=6g-6+2n$ and, therefore, the rank of the 
Drinfeld module is equal to $r=3g-3+n$.

\bigskip
Lemma \ref{lm3.1} is proved.
\end{proof}
%*********************************************************

 %*******************************************************************
\begin{figure}
%*******************************************************************
\begin{picture}(300,110)(-70,0)
\put(20,70){\vector(0,-1){35}}
\put(122,70){\vector(0,-1){35}}
\put(45,23){\vector(1,0){60}}
\put(45,83){\vector(1,0){60}}
\put(15,20){$I_{\alpha}$}
\put(118,20){$\mathbb{A}_p(S_{g,n})$}
\put(15,80){$A$}
\put(115,80){$k\langle\tau_p\rangle$}
\put(5,50){$F$}
\put(130,50){$F$}
\put(55,30){\sf inclusion}
\put(40,90){$r=3g-3+n$}
\end{picture}
%******************************************************************
\caption{}
\end{figure}
%*******************************************************************

%************************************************************
\begin{lemma}\label{lm3.2}
The diagram in Figure 1 is commutative. In particular, 
for each $g\ge 1$ and $n\in\{0,1,2\}$ the  Drinfeld module  $Drin_A^{3g-3+n}(k)$
gives rise to a  noncommutative torus $\mathscr{A}_{RM}^{6g-6+2n}$,
such that isogenous  (isomorphic, resp.) 
 Drinfeld modules correspond  to the homomorphic (isomorphic, resp.) noncommutative tori.
\end{lemma}
%********************************************************
\begin{proof}
(i) Indeed, the functor $F$ in Figure 1 acts by the formula $A\mapsto C^*(\rho(A))$ and   $k\langle\tau_p\rangle\mapsto C^*(k\langle\tau_p\rangle)$. 
Lemma \ref{lm3.1} implies that the diagram in Figure 1 is commutative.  

\smallskip
(ii) Functor $F$ extends to such between the Drinfeld modules and noncommutative tori 
via the formula: 
%********************************************************************
\begin{equation}\label{eq3.2}
Drin_A^{3g-3+n}(k)\cong \rho(A)\mapsto \mathbb{A}_p(S_{g,n})/I_{\alpha} \cong \mathbb{A}_{RM}^{6g-6+2n},
\end{equation}
%*******************************************************************
where $K_0(\mathbb{A}_{RM}^{6g-6+2n}) \cong K_0(\mathscr{A}_{RM}^{6g-6+2n})$.

\smallskip
(iii) Let us show that $F(Drin_A^{3g-3+n}(k))$ and $F(\widetilde{Drin}_A^{3g-3+n}(k))$
are  homomorphic (isomorphic, resp.) noncommutative tori,  whenever $Drin_A^{3g-3+n}(k)$
and $\widetilde{Drin}_A^{3g-3+n}(k)$ are isogenous  (isomorphic, resp.) 
 Drinfeld modules. 
Recall that a morphism of Drinfeld modules $f: Drin_A^{3g-3+n}(k)\to \widetilde{Drin}_A^{3g-3+n}(k)$
 is an element $f\in k\langle\tau_p\rangle$ such that $f\rho_a=\widetilde{\rho_a} f$
for all $a\in A$. An isogeny is a surjective morphism with finite kernel. 
Let $I_{\alpha}'$ be an ideal generated by 
$\iota(f)$  and $I_{\alpha}$,  where $\iota(f)$ is the image of the isogeny $f$ 
in the $C^*$-algebra  $\mathbb{A}_p(S_{g,n})$, see Figure 1. 
The map $\mathbb{A}_p(S_{g,n})/I_{\alpha}\to  \mathbb{A}_p(S_{g,n})/I_{\alpha}'$ is a homomorphism. 
In other words, one gets a homomorphism between the  noncommutative tori  $\mathscr{A}_{RM}^{6g-6+2n}$  and 
$\widetilde{\mathscr{A}}_{RM}^{~6g-6+2n}$.   If $f$ is invertible (i.e. the kernel is trivial), one gets an isomorphism 
of the noncommutative tori. 

\medskip
Lemma \ref{lm3.2} is proved. 
\end{proof}
%*********************************************************

\medskip
 Item (i) of Theorem \ref{thm1.1} follows from Lemma \ref{lm3.2} and  $r=3g-3+n$.

\bigskip
(B)  Let us pass to  item (ii) of Theorem \ref{thm1.1}.  We split the proof in a series of lemmas.

%************************************************************
\begin{lemma}\label{lm3.3}
The generators  $u_1,\dots, u_{2r}$ of the algebra $\mathscr{A}_{RM}^{2r}$ satisfy the following quadratic relations:
%***************************************************************************
\begin{equation}\label{eq3.3}
\left\{
\begin{array}{lll}
u_2u_1 &=& \log ~(\varepsilon) ~e^{2\pi i\alpha_1} ~u_1u_2, \\
u_3u_2 &=& \log ~(\varepsilon) ~e^{2\pi i\alpha_2} ~u_2u_3, \\
\vdots && \\
u_{2r}u_{2r-1} &=& \log ~(\varepsilon) ~e^{2\pi i\alpha_{2r-1}} ~u_{2r-1}u_{2r}. 
\end{array}
\right.
\end{equation}
%**************************************************************************
\end{lemma}
%********************************************************
\begin{proof}
(i)  Relations (\ref{eq3.3}) with  $\log ~(\varepsilon) =1$ follow from formulas (\ref{eq1.1}) and (\ref{eq1.2}).
Let us prove  (\ref{eq3.3}) in the general case.

\smallskip
(ii) Recall that the Tomita-Takesaki flow  $\sigma_t: \mathbb{A}(S_{g,n})\to \mathbb{A}(S_{g,n})$ gives rise 
to a one-paramter group of automorphisms $\varphi^t: \mathbb{A}_{RM}^{2r}\to  \mathbb{A}_{RM}^{2r}$
defined by the formula   $\varphi^t(\mathbb{A}_{RM}^{2r})=\mathbb{A}(S_{g,n})/\sigma_t(I_{\alpha})$;
we refer the reader to Section 2.2.1 and Theorem \ref{thm2.2} for the notation and details. 
It is not hard to see,  that the action of $\varphi^t$ extends to  $\mathscr{A}_{RM}^{2r}$,
where $K_0(\mathscr{A}_{RM}^{2r}) \cong K_0(\mathbb{A}_{RM}^{2r})$.

\smallskip
(iii) To establish the commutation relations  (\ref{eq1.1}) for  $\varphi^t(\mathscr{A}_{RM}^{2r})$,
we must solve the equation:
%*********************************************************************************
\begin{equation}\label{eq3.4}
Tr~(\varphi^t(\mathscr{A}_{RM}^{2r}))=Tr' ~(\mathscr{A}_{RM}^{2r} ), 
\end{equation}
%*********************************************************************************
where $Tr$ is the trace of $C^*$-algebra. It is easy to see, 
that  (\ref{eq3.4}) with $Tr'=t~Tr$ applied to (\ref{eq1.1}) gives  the commutation 
relation for  the $C^*$-algebra  $\varphi^t(\mathscr{A}_{RM}^{2r})$
of the form $u_{k+1}u_k=t e^{2\pi i \alpha_k}u_ku_{k+1}$.
Indeed, equation (\ref{eq3.4}) is equivalent to $Tr~(u_{k+1}u_k)=t ~Tr~(e^{2\pi i\alpha_k} u_ku_{k+1})=Tr~(te^{2\pi i\alpha_k}u_ku_{k+1})$.
The latter is satisfied if and only if  $u_{k+1}u_k=t e^{2\pi i \alpha_k}u_ku_{k+1}$.

\smallskip
(iv) Since our algebra $\mathscr{A}_{RM}^{2r}$ has real multiplication,
the corresponding  Riemann surfaces $S_{g,n}$ and $\phi(S_{g,n})$ must lie at the axis of
a pseudo-Anosov automorphism $\phi\in Mod~(S_{g,n}))$ (Section 2.2.1).  In view of formula (\ref{eq2.6}),  
the Connes invariant is equal to $\log\lambda_{\phi}$, where $\lambda_{\phi}$ is the dilatation of $\phi$.  
In other words, $\varphi^1(\mathscr{A}_{RM}^{2r})\cong \varphi^t(\mathscr{A}_{RM}^{2r})$ if and only if
$t=\log\lambda_{\phi}$.

\smallskip
(v) To express $\log\lambda_{\phi}$ in terms of $\alpha_i$, let $p(x)=x^m-a_{m-1}x^{m-1}-\dots-a_1x-a_0$ be a minimal polynomial of 
the number field $\mathbf{Q}(\alpha_i)$. Consider the  integer matrix 
%*****************************************************************
\begin{equation}\label{eq3.5}
B=\left(
\begin{matrix}
0 & 0 &\dots &0& a_0\cr
1 & 0 &\dots &0& a_1\cr
\vdots & \vdots & && \vdots\cr
0 & 0  & \dots & 1 & a_{m-1}          
\end{matrix}
\right),
\end{equation}
%*********************************************************************************
which one can always assume to be non-negative.  It is well known that $p(x)=\det (B-xI )$
and we let $\varepsilon>1$  be the Perron-Frobenius eigenvalue of $B$.

\smallskip
(vi) Recall that if $m=6g'-6+2n'$ then $B$ corresponds to the action of $\phi$ on the relative cohomology $H^1(S_{g',n'}; \partial S_{g',n'})$ 
of a surface $S_{g',n'}$ [Thurston 1988] \cite[p. 427]{Thu1}.  (Notice that in general $g'\ne g$ and $n'\ne n$,  but  $S_{g',n'}$ is a finite cover
of the surface $S_{g,n}$.)  In particular,  $\lambda_{\phi}=\varepsilon$.  
Thus  one gets the  commutation relations (\ref{eq1.1}) in the form:
%*********************************************************************************
\begin{equation}\label{eq3.6}
u_{k+1}u_k=(\log\varepsilon) ~e^{2\pi i \alpha_k}u_ku_{k+1}=e^{2\pi i \alpha_k+\log\log\varepsilon}u_ku_{k+1}. 
\end{equation}
%*********************************************************************************

\bigskip
Lemma \ref{lm3.3} follows from  (\ref{eq3.6}). 
\end{proof}
%********************************************************

\bigskip
%************************************************************
\begin{lemma}\label{lm3.4}
 $F(\Lambda_{\rho}[a])=\{e^{2\pi i\alpha_k+\log\log\varepsilon} ~|~1\le k\le 2r-1\}$. 
\end{lemma}
%********************************************************
\begin{proof}
(i)  Recall that the Drinfeld module of rank $r=1$ 
can be thought as an elliptic curve $\mathcal{E}_{CM}\cong\mathbf{C}/O_k$ 
with complex multiplication by the ring of integers $O_k$ 
of an imaginary quadratic field $k$ given by the homomorphism (\ref{eq1.3})
of the form $\rho: O_k\to End ~\mathcal{E}_{CM}$,
where $End ~\mathcal{E}_{CM}$ is the endomorphism ring of $\mathcal{E}_{CM}$ [Drinfeld 1974] \cite[p. 594]{Dri1}. 
Moreover, the torsion submodule $\Lambda_{\rho}[a]$ consists of the coefficients of curve $\mathcal{E}_{CM}$,
i.e. the values of the Weierstrass $\wp$-function at the torsion points of the lattice $O_k$; hence the name.

 \smallskip
(ii)  When the rank of the Drinfeld module is $r\ge 1$, a substitute for the $\mathcal{E}_{CM}$ 
is provided by the noncommutative torus  $\mathscr{A}_{RM}^{2r}$. The latter is given by the
equations (\ref{eq3.3}) with the constant terms $e^{2\pi i\alpha_k+\log\log\varepsilon}$. 
We conclude therefore that   $F(\Lambda_{\rho}[a])=\{e^{2\pi i\alpha_k+\log\log\varepsilon} ~|~1\le k\le 2r-1\}$.

\medskip
Lemma \ref{lm3.4} is proved. 
\end{proof}
%********************************************************

\medskip
 Item (ii) of Theorem \ref{thm1.1} follows from Lemma \ref{lm3.4} with $r=3g-3+n$.

\bigskip
(C)  Let us prove item (iii) of Theorem \ref{thm1.1}.  Recall that Theorem \ref{thm2.3} says 
that there exists a subgroup $G\subseteq GL_r(A/aA)$ acting transitively on the
elements of the torsion submodule $\Lambda_{\rho}[a]$.  By item (ii) of Theorem \ref{thm1.1}
the action of $G$ extends to the set  $F(\Lambda_{\rho}[a])=\{e^{2\pi i\alpha_k+\log\log\varepsilon} ~|~1\le k\le 2r-1\}$.
Let $\mathbf{K}=\mathbf{Q}(e^{2\pi i\alpha_k+\log\log\varepsilon})$ be a number field and 
let $\mathbf{k}$ be the maximal subfield of $\mathbf{K}$ fixed by all elements of the group $G$. 
It is easy to see that $\mathbf{K}$ is a Galois extension of $\mathbf{k}$ and 
the Galois group $Gal~(\mathbf{K}|\mathbf{k})\cong G\subseteq GL_r(A/aA)$. 
 Item (iii) of Theorem \ref{thm1.1} follows from the above remark. 
 
 \bigskip
 Theorem \ref{thm1.1} is proved.

%**************************************************************************
\subsection{Proof of Corollary \ref{cor1.2}}
%***************************************************************************
We split the proof in two lemmas. 

%***************************************************
\begin{lemma}\label{lm3.5}
If $\mathbf{k}\subset\mathbf{C}$, then 
$Gal \left( \mathbf{k}\left(e^{2\pi i\alpha_k+\log\log\varepsilon}\right) | ~\mathbf{k} \right)\cong G\subseteq GL_r(A/aA)$. 
\end{lemma}
%***************************************************
%**************************************************
\begin{proof}
If $\mathbf{k}\subset\mathbf{C}$, then $\mathbf{K}$ is a Galois extension of $\mathbf{k}$ only if its generators are complex numbers $\mathbf{C}-\mathbf{R}$. 
We can now apply item (iii) of Theorem \ref{thm1.1} which says that  $Gal \left( \mathbf{k}\left(e^{2\pi i\alpha_k+\log\log\varepsilon}\right) | ~\mathbf{k} \right)\cong G\subseteq GL_r(A/aA)$.
Lemma \ref{lm3.5} is proved. 
\end{proof}
%*************************************************

%***************************************************
\begin{lemma}\label{lm3.6}
If $\mathbf{k}\subset\mathbf{R}$, then 
$Gal \left( \mathbf{k}\left(\cos 2\pi\alpha_k \times\log\varepsilon\right) | ~\mathbf{k} \right)\cong G\subseteq GL_r(A/aA)$. 
\end{lemma}
%***************************************************
%**************************************************
\begin{proof}
(i) If $\mathbf{k}\subset\mathbf{R}$ is a Galois extension of $\mathbf{Q}$, then $\mathbf{K}$ is a Galois extension of $\mathbf{k}$ if and only if  all generators of $\mathbf{K}$ are real numbers.
(This follows from the fact that whenever an extension is Galois over $\mathbf{Q}$,  then it is either totally real or totally imaginary.)
Thus one cannot apply  item (iii) of Theorem \ref{thm1.1},  but  instead one can use  the following natural reduction.   
Recall that a real $C^*$-algebra is a Banach $^*$-algebra $\mathscr{A}$ over
$\mathbf{R}$  isometrically $^*$-isomorphic to a norm-closed $^*$-algebra of bounded
operators on a real Hilbert space, see  [Rosenberg 2016] \cite{Ros1} for an excellent introduction. 
Given a real $C^*$-algebra $\mathscr{A}$, its complexification  $\mathscr{A}_{\mathbf{C}} = \mathscr{A} + i\mathscr{A}$  is a
complex $C^*$-algebra. Conversely, given a complex $C^*$-algebra  $\mathscr{A}_{\mathbf{C}}$, whether there 
exists a real $C^*$-algebra $\mathscr{A}$ such that $\mathscr{A}_{\mathbf{C}} = \mathscr{A} + i\mathscr{A}$ is unknown in general. 
However, the noncommutative torus admits the unique canonical decomposition:
%*********************************************************************************
\begin{equation}\label{eq3.7}
\mathscr{A}_{\Theta}^m=(\mathscr{A}_{\theta}^m)^{Re}+i(\mathscr{A}_{\Theta}^m)^{Re}. 
\end{equation}
%*********************************************************************************

\medskip
(ii)  Let  $\mathscr{A}_{RM}^{2r}=(\mathscr{A}_{RM}^{2r})^{Re}+i(\mathscr{A}_{RM}^{2r})^{Re}$. 
It is easy to verify, that the commutation relations (\ref{eq3.3}) for the real $C^*$-algebra $(\mathscr{A}_{RM}^{2r})^{Re}$ 
have the form: 
%*********************************************************************************
\begin{equation}\label{eq3.8}
u_{k+1}u_k=Re \left(e^{2\pi i \alpha_k+\log\log\varepsilon}\right)u_ku_{k+1}= \left(\cos 2\pi\alpha_k \times\log\varepsilon\right)u_ku_{k+1}. 
\end{equation}
%*********************************************************************************
One concludes from (\ref{eq3.8}) that $\mathbf{K}=\mathbf{k}(\cos 2\pi\alpha_k \times\log\varepsilon)$,
where $\mathbf{k}\subset\mathbf{R}$.
In view of item (iii) of Theorem \ref{thm1.1},  we have 
$Gal \left( \mathbf{k}\left(\cos 2\pi\alpha_k \times\log\varepsilon\right) | ~\mathbf{k} \right)\cong G\subseteq GL_r(A/aA)$. 
This argument finishes the proof of Lemma \ref{lm3.6}. 
\end{proof}
%***************************************************

\bigskip
Corollary \ref{cor1.2} follows from Lemmas \ref{lm3.5} and \ref{lm3.6}.

%**************************************************************************
\section{Langlands program for noncommutative tori}
%***************************************************************************
The Langlands Program predicts  an analytic solution to the non-abelian class field theory based on the $L$-functions
associated to the irreducible representations of the algebraic groups over adeles. 
Namely,   the $n$-dimensional Artin $L$-function $L(s,\sigma_n)$  [Artin 1924] \cite{Art1} is conjectured to coincide with an $L$-function 
$L(s,\pi)$,  where $\pi$ is a representation of the  adelic algebraic group $GL(n)$ [Langlands 1978] \cite{Lan1}.

The $L$-functions  $L(s, \mathscr{A}_{RM}^{2n})$ of  the even dimensional noncommutative tori $\mathscr{A}_{RM}^{2n}$
were introduced  in \cite{Nik1}, see also \cite[Section 6.6]{N}.  
It was proved that the Artin $L$-function  $L(s,\sigma_n)\equiv L(s, \mathscr{A}_{RM}^{2n})$ if  $n=0$ or  $n=1$,  and
conjectured to be true for all $n\ge 2$ and a suitable choice of the $\mathscr{A}_{RM}^{2n}$ \cite[Conjecture 6.6.1]{N}.

 Recall that the Artin $L$-functions are well defined for the non-abelian groups $G$ [Artin 1924] \cite{Art1}.
 On the other hand, the  Galois extensions with the  group $G$ are classified by the noncommutative tori
 $\mathscr{A}_{RM}^{2n}$ (Theorem \ref{thm1.1} and Corollary \ref{cor1.2}). 
 %*******************************************************
\begin{conjecture}
Let  $\mathscr{A}_{RM}^{2n}$  be a noncommutative torus underlying the extension (\ref{eq1.5}) 
with the Galois group $G$. Let  $\sigma_n: G\to GL_n(\mathbf{C})$  be a representation of $G$ and $L(s,\sigma_n)$ be 
the corresponding Artin $L$-function. Then  for all $n\ge 1$ the algebra $\mathscr{A}_{RM}^{2n}$ solves  the identity:
%*******************************************************************
%\begin{equation}\label{eq4.1}
\displaymath
L(s,\sigma_n)\equiv L(s, \mathscr{A}_{RM}^{2n}).
\enddisplaymath
%\end{equation}
%****************************************************************** 
 \end{conjecture}
%********************************************************

%********************************
\section*{Declarations}
%********************************
The author states that there is no ethical approval, funding or data availability linked to the manuscript.

\subsection*{Acknowledgments}
 The author would like to thank the anonymous referee who provided thoughtful  comments on an earlier version of the manuscript.

\bibliographystyle{amsplain}

%**********************************************************

\end{document}